\documentclass[leqno,11pt]{amsart}

\usepackage{graphicx,color}
\usepackage{latexsym}
\usepackage{graphicx} 
\usepackage{amsthm,amsmath, amssymb,amsfonts,esint}
\usepackage{layout,textcomp}

\usepackage[english]{babel}

\newcommand{\Div}{\mbox{div}}
\newcommand{\tr}{\mbox{tr}}
\newcommand{\Ric}{\mbox{Ric}}
\newcommand{\Hess}{\mbox{Hess}}
 
\newtheorem{theorem}{Theorem}[section]

\newtheorem{lemma}[theorem]{Lemma}

\newtheorem{corollary}[theorem]{Corollary}

\newtheorem{proposition}[theorem]{Proposition}

\theoremstyle{remark}
\newtheorem{remark}[theorem]{Remark}

\numberwithin{equation}{section}

\begin{document}

\title[Prescribing the curvature of  manifolds with boundary]{ Prescribing the curvature of Riemannian manifolds with boundary}

\author{ Tiarlos Cruz}
\author{Feliciano Vit\'orio}
\address{Universidade Federal de Alagoas\\
Instituto de Matemática\\
Maceió, AL - Brazil 57072-970, Brazil}
\email{cicero.cruz@im.ufal.br}
 \email{feliciano@pos.mat.ufal.br}


\begin{abstract}

Let $M$ be a compact connected surface with boundary. 
We prove
that the signal condition given by the Gauss-Bonnet theorem  is  necessary and sufficient  for a given smooth  
function $f$ on $\partial M$ (resp. on $M$) to be  geodesic curvature of the boundary (resp. the Gauss curvature)  of some flat metric on $M$ (resp. metric on $M$ with geodesic boundary). 
In order to provide  analogous results for this problem with $n\geq 3,$ we prove some topological restrictions  which  imply, among other things, that  any function that is negative somewhere on $\partial M$ (resp. on $M$) is a mean curvature  of a scalar flat metric on $M$ (resp.   scalar curvature of a metric on $M$ and minimal  boundary with respect to this metric).
As an application of our results, we obtain a classification theorem for manifolds with boundary.
\end{abstract}

\maketitle

\section{Introduction}

A natural  problem in differential geometry is to find metrics with prescribed curvature, i.e, construct a Riemannian metric on a given smooth manifold $M$ whose  curvature is equal to a given function $f$ on $M$.

  On closed manifolds, the prescribed Gaussian (resp. scalar) curvature problem has been completely solved by Kazdan and Warner \cite{KW, KW2,KW3}. 
Here we address this problem for manifolds with boundary.
For instance,  let $M$ be a surface with  boundary $\partial M,$ given  a smooth  function $h$ defined on the boundary (or $f$ in the interior), is there a Riemannian metric $g$ such that the geodesic curvature  $\kappa_{g_{\partial M}}=h$ (or Gaussian curvature $K_{g}=f)?$ In fact  such a problem is equivalent to solve a quasilinear system of partial differential equation 
 with boundary conditions.
We point out that, as a particular case, it is possible to solve this problem by  conformal deformation of the metric, which  consists in  picking some metric $g_0$ on $M$ and seeking  a conformally related metric to $g_0,$ say  $g=e^{2u}g_0,$ for some positive function $u$ to be found in order to satisfy 
\begin{equation*}
\left\{
   \begin{array}{rcl}
-\Delta_{g_0} u+2K_{g_0}&=&2fe^u\quad\mbox{ in}\quad M\\
\frac{\partial u}{\partial \nu}+2\kappa_{g_0}&=&2he^\frac{u}{2}\quad\mbox{on}\quad \partial M,
  \end{array}
   \right.
\end{equation*}
where $\Delta_{g_0},$ $K_{g_0}$ and  $\kappa_{g_0}$  are the Laplacian, the Gauss curvature and the geodesic curvature of the boundary of $g_0,$ respectively. Here $\nu$ is the outward unit normal on $\partial M.$ On this subject, the literature is extensive and many  results are known, see  for instance  \cite{E4,EG,S}, that includes the higher dimensional case as  well as the recent works \cite{LR, LMR,CR}.

Suppose  $M$ is a compact two-dimensional Riemannian manifold with boundary $\partial M$. The Gauss-Bonnet theorem states that
$$
\int_{M} Kdv+\int_{\partial M}\kappa d\sigma=2\pi\chi(M),
$$
where $K$ denotes the Gaussian curvature, $\kappa$ is the geodesic curvature of the boundary, $\chi(M)$ is the Euler characteristic, $dv$ is  the element of volume and $d\sigma$ is the element of area. Besides establishing  a link between the topology (Euler characteristic) and geometry of a surface, it also  
gives a necessary  signal condition on the Gaussian curvature of a  surface  or  geodesic curvature on the boundary in terms of its   Euler characteristic.

Consider the following natural consequence given by the Gauss-Bonnet theorem when $M$ is a  bounded domain $\Omega$ in $\mathbb R^2$  with smooth boundary (resp. compact connected 2-manifold with geodesic boundary):
\begin{eqnarray}\label{C1}
& & \mbox{If $\mathcal{X}(M)>0,$ then $\kappa$ (resp. $K$) must be positive somewhere.} \nonumber\\
& & \mbox{If $\mathcal{X}(M)=0,$ then $\kappa$ (resp.  $K$) must change sign unless it is  $\kappa\equiv0$.}\\
& &\mbox{If  $\mathcal{X}(M)<0,$ then $\kappa$ (resp. $K$)  must be negative somewhere. }\nonumber
\end{eqnarray}

In the first result we prove that the obvious signal condition (\ref{C1}) is also sufficient to the problem of prescribing  curvature. More precisely, we state the following theorem:

\begin{theorem}\label{teoa}
Let $\Omega\subset \mathbb{R}^2$ be a bounded domain  with smooth boundary.  
A function $\kappa\in C^{\infty}(\partial \Omega)$ is the geodesic curvature of a flat metric on $\Omega$ if only if $\kappa$ satisfies the signal condition (\ref{C1}).
\end{theorem}

We also prove the following result for manifolds with geodesic boundary.

\begin{theorem}\label{teob}
Let $M^2$ be a compact connected surface with smooth geodesic boundary.  
A function $K\in C^{\infty}(M)$ is the Gaussian  curvature of a metric on $M$ with geodesic boundary if only if $K$  satisfies the signal condition (\ref{C1}).
\end{theorem}

One of the key ingredients in the proof of Theorem \ref{teoa} and Theorem \ref{teob} is the celebrated  Osgood, Phillips, and Sarnak  uniformization theorem for surfaces with boundary  \cite{OPS}(see also Brendle \cite{Br,Br1}). Namely, if the surface has boundary, in each conformal class of Riemannian metrics, there is a unique uniform metric of type I, i.e.,  a constant curvature metric with zero geodesic curvature, and a unique uniform metric of type II, i.e. the resulting Riemannian manifold $M$ is flat, i.e. the sectional curvature is zero and $\partial M$ has constant geodesic curvature on  the boundary.

In order to generalize Theorem \ref{teoa} and Theorem \ref{teob} we make use of a version of the uniformization theorem in higher dimensions. In this respect,  we have the Yamabe problem for manifolds with boundary  that consists in finding a  conformal metric to the background  one having constant scalar curvature and  minimal boundary or having zero scalar curvature and  constant mean curvature on $\partial M.$ Such a problem  has inspiration in the closed case and it was solved  in almost every case by  
Escobar \cite{E1,E2}. We refer  the interested reader to  Marques \cite{M1,M2},  Almaraz \cite{A},  Brendle and Chen \cite{B} and Mayer and Ndiaye \cite{MN} that studied many of the remaining cases.

Using results of existence of metrics
with constant scalar curvature  and minimal boundary or with zero  scalar curvature whose boundary has constant mean curvature,
 we have the following theorems. 

\begin{theorem}\label{corA}
Let $M^n$, $n\geq 3$, be a compact connected manifold with  boundary.
\begin{itemize}
\item[i)]Any function  on $\partial M$  that is negative somewhere is a mean curvature  of a scalar flat metric on $M$.
\item[ii)]Every  smooth function on $\partial M$  is a mean curvature of a scalar flat metric if and only if $ M$ admits a scalar flat metric with positive constant mean curvature on the boundary.
\end{itemize}
\end{theorem}

\begin{theorem}\label{corB}
Let $M^n$, $n\geq 3$, be a compact connected manifold with smooth boundary.
\begin{itemize}
\item[i)]Any function on $M$  that is negative somewhere  is a scalar curvature of a metric  with minimal boundary.
\item[ii)]Every  smooth function on $M$  is a scalar curvature of a metric with minimal  boundary with respect to this metric if and only if $M$ admits a metric
with positive constant scalar curvature and minimal boundary.
\end{itemize}
\end{theorem}

Taking account some  topological restrictions given in Section \ref{restr},  we separate the compact manifolds with boundary
into three groups:

\begin{theorem}\label{classes} 
 Compact manifolds with  boundary and dimension $n\geq 3$  can be divided into three classes:
\begin{itemize}
\item[a)] Any smooth function on $\partial M$ (resp. $M$) is mean curvature  of some  scalar flat metric (resp.   scalar curvature of a metric on $M$ with minimal the boundary with respect to this metric);
\item[b)] A smooth function on $\partial M$ (resp. $M$) is mean curvature  of some  scalar flat metric on $M$ (resp.   scalar curvature of a metric with minimal  boundary with respect to this metric) if and only if it is either identically equal to zero or strictly negative somewhere; furthermore,  any scalar flat metric having zero mean curvature  is totally geodesic (resp. Ricci-flat).
\item[c)] A smooth function on $\partial M$ (resp. $M$) is mean curvature  of some scalar flat  metric 
(resp. scalar curvature of a metric with minimal the boundary with respect to this metric)
if and only if it is strictly negative somewhere.
\end{itemize}
\end{theorem}

In short, every  compact manifold with  boundary of dimension $n\geq 3$ admits a  scalar flat metric on $M$ with constant negative mean curvature on $\partial M$ (resp. a metric with constant negative scalar curvature and minimal boundary). Those in item $a)$ or $b)$ are scalar flat  on $M$ and have vanishing mean curvature on the boundary, and those in item $a)$ are scalar flat on $M$ and have constant positive mean curvature on the boundary (resp. constant positive scalar curvature and minimal boundary).

This paper is organized as follows. In Section \ref{preli}, we gather some important preliminary tools, discuss notations
and  formally present the second order linear operator we shall study. 
  In Section \ref{LS},  we prove the remarkable property that the map  $g\mapsto (R_g,2H_{g_{\partial M}})$ is  almost always a surjection, which,  together an approximation lemma contained in Section \ref{pres},  allow us to prove in Section \ref{presc} and \ref{restr}   results concerning what functions can be realized as scalar curvature or mean curvature of the boundary for dimension $n\geq2.$ To be more precise, we prove  Theorem  \ref{teoa} and \ref{teob} in Section \ref{presc} and, discussing some topological obstructions results, we prove Theorem \ref{corA}, \ref{corB} and  \ref{classes} in Section \ref{restr}.
\\

\noindent{\it Ackwnoledgement.} 
  The authors wish to express their gratitude to Almir Santos and  Levi Lima for several discussions and a number of enlightening conversations. We would like to thank the referee for carefully reading  our manuscript and for giving such constructive  comments which helped improving the quality of the paper.
  F.V. was partially supported by CNPq-Brazil and FAPEAL-Brazil. T.C. is grateful to the  ICTP - International Centre for Theoretical Physics, where part of this article was written in January/February 2018, and he would like to thank the Department of Mathematics for its hospitality. 

\section{preliminaries}\label{preli}

Let $M^n$ be $n$-dimensional  compact connected  Riemannian manifold with boundary. 
Let $S^{2,p}_2=W^{2,p}(\mbox{Sym}^2(T^*M))$ denote the section of class $W^{2,p}$ of  symmetric $(0, 2)$-tensors.  For $p>n,$ consider the operator 
$$
\Psi(\cdot):=(R(\cdot), 2H(\cdot)):\mathcal{M}^{2,p}\to L^p(M)\oplus W^{\frac{1}{2},p}(\partial M),
$$
 where  $\mathcal{M}^{2,p}$ denotes the open subset of $S^{2,p}_2$ of the Riemannian metrics on $M.$  Since $R_g$ and  $H_{g_{\partial M}}$  involve derivatives of $g$ up to second order,  by the Sobolev Embedding Theorem, for $p>n$, we have that  $\Psi$  is a $C^{\infty}$ map. 
 Given an infinitesimal variation $h.$ We introduce 
 $$
\delta R_g h:=\frac{\partial}{\partial t}R_{g+th}\Big|_{t=0} 
 $$
 and 
 $$
\delta H_{\gamma} h:=\frac{\partial}{\partial t}H_{\gamma+th}\Big|_{t=0},
 $$ 
 the variation of the scalar curvature $R$  and of the mean curvature $H$ in the direction of $ h$, respectively.  Here  $\gamma=g_{\partial M}.$ 
A classical computation, that can be found in \cite{HA}, shows that 
\begin{equation}\label{prob}
\left\{
   \begin{array}{rcl}
\delta R_g h&=& -\Delta_g(\tr_gh)+\mbox{div}_g \mbox{div}_g h - \langle h, \Ric_g\rangle\\
2\delta H_{\gamma} h&=&[ d ( \tr_ g h ) - \mbox{div}_g h ] (\nu) - \Div_\gamma X - \langle \Pi_{\gamma}, h \rangle_\gamma   \end{array}
   \right.
\end{equation}
where $ \nu $ is the outward unit normal to $ \partial M$,  $\Pi_{\gamma}$ is the second fundamental form of $\partial M,$ 
$X$ is the vector field dual to the one-form $\omega(\cdot)=h(\cdot,\nu),$  $\tr_g=g^{ij}h_{ij}$ is the trace of $h$  and  our convention for the laplacian is $\Delta_g f = \text{tr}_g(\mbox{Hess}_g f)$.
The  linearization of $\Psi$ will be denoted by $$\mathcal{S}_g(h)=D\Psi_g\cdot h=(\delta R_g h, 2\delta H_{\gamma} h).$$

Before proceeding, we need of the following lemma. 

\begin{lemma}\label{lemmadiv}
 Let $h$ be a symmetric $(0,2)$-tensor and $f$ be a smooth function on  $M.$
Then
\begin{equation}
\int_M f \mbox{div}_g\mbox{div}_g h\;dv=\int_M \langle \Hess_gf, h\rangle\;dv+\int_{\partial M}f\langle\mbox{div}_g h, \nu\rangle-h(\nabla f,\nu)\;d\sigma
\end{equation}

\end{lemma}

The proof of Lemma \ref{lemmadiv} is just to apply the divergence theorem to the field  $X = f\mbox{div}_g\;h - h(\nabla f,\cdot).$

A direct calculation using the previous lemma gives that

\begin{eqnarray*}\label{eqdepoisdiv}
\int_M f\delta R_gh+2\int_{\partial M} f\delta H_{\gamma} h &=&\int_M(-\Delta_g f(\tr_gh)+\langle \Hess_gf, h\rangle - f\langle h, \Ric_g\rangle)\nonumber\\
& +& \int_{\partial M}  \tr_g h\frac{\partial f}{\partial \nu}- f\langle \Pi_{\gamma}, h \rangle_\gamma -f\Div_\gamma X-\omega(\nabla f)\\
 &=&\int_M(-\Delta_g f(\tr_gh)+\langle \Hess_gf, h\rangle - f\langle h, \Ric_g\rangle)\nonumber\\
&+ & \int_{\partial M}  \tr_\gamma h\frac{\partial f}{\partial \nu}- f\langle \Pi_{\gamma}, h \rangle_\gamma,
\end{eqnarray*}
where we have omitted the volume forms and used the fact that $$-f\Div_\gamma X=\omega(\nabla f)-h(\nu,\nu)\frac{\partial f}{\partial \nu}.$$
We first observe that the previous calculations clearly shows  that  $$
(\delta R_g h,f)_{L^2(M)}-(A^*f,h)_{L^2(M)}=
(B^*f,h)_{L^2(\partial M)}-2(\delta H_{\gamma} h,f)_{L^2(\partial M)}
$$
for all $h\in\mathcal{M}^{2,p}$ and $f\in W^{2,p}(M),$ where $$(u,v)_{L^2(M)}=\int_{M}uv\;dv$$ and 
  \begin{eqnarray}
  A_g^*f&=& -(\Delta_g f)g+\mbox{Hess}_gf-f\Ric_g\;\;\;\mbox{in}\;\;\;M\nonumber \\
  B_\gamma^*f &=&\frac{\partial f}{\partial \nu}\gamma-f\Pi_{\gamma}\;\;\;\;\;\;\;\;\;\;\;\; 
  \;\;\; \;\;\;\;\;\;\;\; \;\;\; \mbox{on}\;\;\;\partial M.\nonumber 
  \end{eqnarray}

Therefore, we introduce the formal  $L^2$-adjoint   of $\mathcal{S}_g(h)$ to be the operator $\mathcal{S}_g^*:W^{2,p}(M)\to S^{0,p}_2(M)\oplus S_2^{\frac{1}{2},p}(\partial M)$  given by 
\begin{equation}\label{static}
\mathcal{S}_g^*(f)=(A_g^*f,B_\gamma^*f), 
\end{equation}

We claim that  $\mathcal{S}_g^*(f)$  is an underdetermined elliptic  operator, which in other words means that it has injective symbol (see \cite{H} for definition). Indeed,  the  principal symbol of $\mathcal{S}^*$ is given by 
$$
\left\{
   \begin{array}{rcl}
\sigma(A_g^*)_x(\varepsilon)f&=&(\|\varepsilon\|^2g-\varepsilon\otimes\varepsilon)f,\\
\sigma(B_\gamma^*)_x(\epsilon)f&=&\langle \epsilon,\nu\rangle_\gamma f\gamma,
 \end{array}
   \right.
$$
for all $\varepsilon\in T^*_xM$ (cotangent space at $x$) and  $\epsilon$ normal to $\partial M$ at $x.$ 
Note that  the principal symbol of $A_g^*$  is injective for $\varepsilon\neq0.$ To see this, if we assume that $\sigma(A_g^*)_x(\varepsilon)f=0,$ taking its trace,  we see that  $(n-1)|\varepsilon|^2f=0.$ 
Moreover, for every linearly independent couple of vectors $\epsilon$ and $\eta$ belonging to $\mathbb {R}^n$, the polynomial 
 in the complex variable $\tau$  
$$
\sigma(A_g^*)(\epsilon+\tau\eta)=g(\|\epsilon\|^2+\tau^2\|\eta\|^2)-[(\epsilon+\tau\eta)\otimes(\epsilon+\tau\eta)]
$$
  has exactly  two roots, one with positive and one with negative imaginary part (Indeed, take the
trace and  obtain $0 = (n-1)(\|\eta\|^2\tau^2+\|\epsilon\|^2)$). Thus,  $A_g^*$ is  a second order  (overdetermined) elliptic operator.   One  condition has to be satisfied by $\mathcal{S}_g^*(f)$  to be an elliptic boundary problem\footnote{Sometimes called of  oblique derivative problem}, that is $A_g^*$ to be elliptic on $M$, and
properly elliptic, and  $B_\gamma^*$ to satisfy the {\it Shapiro-Lopatinskij condition}  at any
point of the boundary, for precise definition see Section 20.1 of \cite{H}, Section 2.18 of \cite{ES} or chapther 2 of \cite{LM}.
Since  $\nu$ is not tangent to $\partial M,$ it is possible to verify that the boundary problem satisfies the Shapiro-Lopatinskij condition (see for example ($Ell_2$) in  \cite{ES}, page 108, as  well all discussion in p. 107-108 and Theorem 2.50 of \cite{ES} that relates  elliptic boundary  properties  such as Fredholm operator, regularity of solutions and an a priori estimate).

 \section{Local surjectivity}\label{LS}

  In this section, we study the  surjectivity of the operator $\mathcal{S}_g$ which, in fact, relies on study the injectivity  $\mathcal{S}_g^*,$ i.e.,  we have to analyse the linear partial differential equation $\mathcal{S}_g^*f=0.$
Consider  $f\in Ker \mathcal{S}_g^*.$ Taking the trace we get that
 \begin{equation}\label{trace}
\left\{
  \begin{array}{rcl}
\Delta f+\frac{R_g}{n-1}f&=& 0\;\;\;\mbox{in}\;\;\;M\\
\frac{\partial f}{\partial \nu}-\frac{H_{\gamma}}{n-1}f&=&0\;\;\;\mbox{on}\;\;\;\partial M.
\end{array}
  \right.
 \end{equation}
Therefore, $\mathcal{S}_g^*f=0$ can be rewritten as 

\begin{equation}\label{new}
\Hess_g f=\Big(\Ric_g-\frac{R_g}{n-1}g\Big)f\;\mbox{in}\;\;M\;\;\mbox{and}\;\;
f\Big(\Pi_{\gamma}-\frac{H_{\gamma}}{n-1}\gamma\Big)=0\;\mbox{on}\;\partial M.
 \end{equation}

 In order to prescribe the  curvature we have to show that the following  $g\mapsto (R_g,2H_{\gamma})$ map is  locally  surjective.  In this section our methods are based on those of Fischer and Marsden \cite{F-M} (see also \cite{F-M2} and \cite{C}). We remark that the same conclusions in this section hold under the condition of metrics close in $W^{s,p}$ norm for $s>\frac{n}{p}+1.$

\begin{proposition} \label{prop1} Given $g\in \mathcal{M}^{2,p},$ $p>\mbox{dim}\;M.$  Suppose  that the scalar curvature (resp. Gauss curvature,  if $\mbox{dim}\;M=2$)  in the interior vanishes with respect to $g$ on $M.$  
Then $\mathcal{S}_g:S^{2,p}\to L^{p}(M)\oplus W^{\frac{1}{2},p}(\partial M)$  is a surjection if 
one of the following holds:
\begin{itemize} 
\item[i)] $H_{\gamma}$ (resp. $\kappa_{\gamma},$ if $\mbox{dim}\;M=2$) is not  a positive constant;
\item[ii)] $H_{\gamma}=0$, but  $\Pi_{\gamma}$ is not identically zero.
\end{itemize}  
\end{proposition}

\begin{proof}

Since $\mathcal{S}_g^*$ has  injective symbol,  it suffices to show that $\mathcal{S}_g^*$ is  injective. 
We first prove  part $ii),$  let $f\in Ker\mathcal{S}_g^*=Ker(A_g^*,B_\gamma^*).$ 
It is immediate to see that if $H_{\gamma}=0$ in (\ref{trace}), then  $f$ is a constant function which together with  $f\Pi_{\gamma}=0$ on $\partial M$ implies that $f$ is constant equal to zero.

For part $i),$ assume that $\mbox{dim}\;M>2$ and $f$ is not identically zero on $\partial M$. It follows from (\ref{new}) that $\Pi_{\gamma}=\frac{H_{\gamma}}{n-1}\gamma.$ 
Recall  that if $\{e_i\}_{i=1}^{n-1}\in \mbox{span}T(\partial M)$, where $e_n=\nu,$ then we get by Codazzi equation that 
\begin{eqnarray}\label{codazzi}
\nabla_i^{\partial M}H_{\gamma}&=&\nabla_i^{\partial M}\Pi^j_j=\nabla_j^{\partial M}\Pi^j_i+R^j_{ji\nu}=\nabla_j^{\partial M}\Pi^j_i+R_{i\nu}\nonumber. 
\end{eqnarray}
where $R_{ijkl}$ is the curvature tensor of $(M,g).$ Thus,
 $$
 \nabla_i^{\partial M}H_{\gamma}=\nabla_j^{\partial M}\Big(\frac{H_{\gamma}}{n-1}\gamma_i^j\Big) +R_{i\nu}=\frac{1}{n-1}\nabla_i^{\partial M}H_{\gamma} +R_{i\nu}.
 $$
However, from $B^*_{\gamma}f=0$ we have that 
\begin{eqnarray}\label{front}
0&=& \nabla_i\Big(\frac{\partial f}{\partial \nu}\gamma^i_j-f\Pi_j^i\Big)\\
&=&\nabla_i\nabla_\nu f-\frac{f}{n-1}\nabla_i^{\partial M}H_{\gamma}.\nonumber
\end{eqnarray}
Since $R_g=0$ (and so $\Delta_g f=0$ in $M$) and $A_g^*f=0,$ we have that 
$\Hess_g f=f\Ric_g,$ which together with all the above facts  implies that  $\nabla^{\partial M} H_{\gamma}=0,$ so $H_{\gamma}$ is constant.
 Notice that  (\ref{trace}) implies that $H_{\gamma}$ is an eigenvalue of the Steklov problem  of second order and hence  $H_{\gamma}> 0,$ contradicting  $i).$

If $\mbox{dim}\;M=2$ and  $f$ is not identically zero on $\partial M,$ then (\ref{front}) is automatically equivalent to $\nabla_i^{\partial M}\kappa_\gamma=0$ provided $\Hess_g f=fK_gg=0.$ Arguing similarly, we conclude that $f$ has to be zero and  $\mathcal{S}_g^*$ is  injective.

We claim that if  there exists a point  $x_0\in \partial M$  so that $f(x_0)=0,$  then we must have $\nabla^{\partial M} f(x_0)\neq 0.$ Reasoning by contradiction, assume that $\nabla^{\partial M} f(x_0)=0.$  We define  $ h(t) := f(\alpha(t)),$ where $\gamma$ 
 is any geodesic on the boundary  of  $M$ starting from $x_0$ . A direct calculation gives that $h(t)$ satisfies the following linear second-order ODE:
  \begin{eqnarray*}
h''(t)&=& (\Hess_g f{|_{\partial M}})_{\alpha(t)}\cdot(\alpha'(t),\alpha'(t))\\
&=&(\Hess_g f)_{\alpha(t)}\cdot(\alpha'(t),\alpha'(t))+\langle\nabla (f\circ \alpha),\Pi_{\gamma}(\alpha'(t),\alpha'(t))\rangle\\
& =& \Big(\Ric_g(\alpha'(t),\alpha'(t))-\frac{R_g}{n-1}\|\alpha'(t)\|^2_g+\frac{H_{\gamma}}{n-1}\Pi_{\gamma}(\alpha'(t),\alpha'(t))\Big)f\circ \alpha,
\end{eqnarray*}
where we have used (\ref{new}).
Because $h(0)=0$ and $h'(0)=0,$ we have that $h(t)\equiv 0$ and, thus, $f=0$  on $\partial M.$

As consequence,  $0$ is a regular
value of $f_{|_{\partial M}}$  which  implies that $\nabla^{\partial M} f\neq0$ on $f^{-1}(0)$ and $\nabla^{\partial M} H_{\gamma}=0$ on an open dense set and hence everywhere. As before, we obtain another contradiction. Thus,  $f$ is zero on $\partial M$ and so  $\mathcal{S}_g^*$ is injective. 

 \end{proof}

\begin{remark}
If $\mbox{dim}\;M=2,$ then item $ii)$ of  Proposition \ref{prop1} (or the second item  of the next proposition) cannot hold.
\end{remark}

\begin{proposition} \label{prop2}
 Given $g\in \mathcal{M}^{2,p},$ $p>\mbox{dim}\;M.$  Suppose  that the mean (resp. geodesic,  if $\mbox{dim}\;M=2$) curvature of $\partial M$ vanishes with respect to the metric $g.$
Then $\mathcal{S}_g:S^{2,p}\to L^{p}(M)\oplus W^{\frac{1}{2},p}(\partial M)$  is a surjection if  one of the following holds:
\begin{itemize} 
\item[i)]  $R_g$  (resp. $K_g=0$,  if $\mbox{dim}\;M=2$) is not  a positive constant;
\item[ii)] $R_g=0$, but  $\Ric_g$ is not identically zero.
\end{itemize} 
\end{proposition}	

\begin{proof}
Due to the similarity of the arguments, we merely sketch the proof. The condition in item $ii)$ implies that 
$f$ is constant, in fact, equal to zero provided $f\Ric_g=0.$ 

We now sketch item $i)$. Since   $$0=\mbox{div}_g A^*f=f\nabla R_g\quad\mbox{on}\quad M\quad\mbox{and}\quad0=\mbox{tr}_g B^*f=\frac{\partial f}{\partial\nu}\quad\mbox{on}\quad \partial M$$ then  $R_g$ is constant if $f$ is not identically zero. However,  note that $R_g$ is constant equal to zero because $\frac{R_g}{n-1}$ is not a positive eigenvalue of the Neumann problem.  We  claim that $\nabla f(x_0)\neq 0$ whenever  $f$ is zero at some point.  Indeed, consider the geodesic $\alpha$ starting at $x_0.$ Defining $h(t)=f\circ\alpha(t),$ a  similar computation shows that $h$, and thus $f$, is identically equal to zero. Hence  we can conclude that $f\in Ker\mathcal{S}^*_g$ is constant equal to zero. 
	
\end{proof}

\begin{remark}
For $n=2,$ the map $\Psi$ cannot be onto a neighborhood of $\Psi(g)=0$. 	For instance,  the Gauss-Bonnet theorem for manifold with boundary shows that a hemisphere does not admit metric with Gauss curvature strictly positive or strictly negative and minimal boundary  as well as an  annulus   with two boundaries components does not admit metric whose Gauss curvature vanishes and the boundary components have geodesic curvature strictly positive or strictly negative.
\end{remark}

Now, for locally solve $\Psi(g)=(f_1,f_2)$ in an appropriate topology, we use the implicit function theorem to prove that $\Psi$ is a  locally surjective map whenever  $\mathcal{S}^*$  is injective. Moreover,  we apply  standard elliptic theory to  $(A_gA_g^*f,B_\gamma B_\gamma^*f).$  Indeed, if  $A^*$ is an operator of order $2$ with injective symbol, then  $A$ has also injective symbol and, thus, $A_gA_g^*$ is elliptic of order $4,$ provided $\sigma(A_gA_g^*)=\sigma(A_g)\sigma(A_g^*)=\sigma(A_g)\sigma(A_g)^*$ and $\sigma(A_g^*)$ injective implies $\sigma(A_g)\sigma(A_g)^*$ is an isomorphism. Moreover, since $B_\gamma$ and $B^*_\gamma$ satisfy the Shapiro-Lopatinskij condition,  $B_\gamma B_\gamma^*$  also satisfies it.

\begin{theorem}\label{teols}
Let $f=(f_1,f_2)\in L^{p}(M)\oplus W^{\frac{1}{2},p}(\partial M)$, $ p>n$. Assume that
 $\mathcal{S}_{g_0}^*$ is injective, then 
 there is an $\eta>0$ such that  if
 $$
 \| f_1-R_{{g_0}}\|_{L^{p}(M)}+\| f_2-H_{\gamma_0}\|_{W^{1/2,p}(\partial M)}<\eta,
 $$ then there is a metric $g_1\in \mathcal{M}^{2,p}$
such that $\Psi(g_1)=f$ Moreover, $g$ is smooth in any open set where $f$ is smooth.
\end{theorem}
 
\begin{proof} In order to  apply the implicit function theorem we consider  the following operator  $S:U\subset W^{4,p}(M)\to L^{p}(M)\oplus W^{\frac{1}{2},p}(\partial M)$ defined by 
$$
S(u)=(R_{{g_0}+A_{g_0}^*u} , H_{{\gamma_0}+B_{\gamma_0}^*u}),
$$
where $U$ is sufficiently small neighborhood of zero in $W^{p,4}.$ Indeed, this is an oblique boundary value problem for a second order quasilinear elliptic differential equation and by the  Sobolev Embedding Theorem, with $n>p,$  $S$ is a $C^1$ map from  $W^{4,p}(M)$ to $L^{p}(M)\oplus W^{\frac{1}{2},p}(\partial M).$
We claim that $S'(0)$ is an isomorphism when restricted a small neighborhood of  $ W^{4,p}$ norm.
In fact, $S(0)=\Psi(g_0)$ and 
$$
S'(0)v=(A_{g_0} A_{g_0}^*v , B_{\gamma_0}B_{\gamma_0}^*v)=\mathcal{S}_{g_0}\mathcal{S}_{g_0}^*v.
$$
 Hence $\mbox{Ker}\;S'(0)=\mbox{Ker}\;\mathcal{S}_{g_0}\mathcal{S}_{g_0}^*=0$, provided  $Ker\;\mathcal{S}_g\mathcal{S}_g^*=Ker\;\mathcal{S}_g^*$.
It follows from the implicit function theorem  that $S$ maps a
neighborhood of zero in $W^{p,4}$ onto a neighborhood of $S(0)=\Psi(g_0)$ in $L^p(M)\times W^{\frac{1}{2},p}(\partial M)$. Considering $L^{p}(M)\oplus W^{\frac{1}{2},p}(\partial M)$ with the norm 
$$
\|(h,w)\|_{L^{p}(M)\oplus W^{1/2,p}(\partial M)}=\|h\|_{L^{p}(M)}+ \|w\|_{W^{1/2,p}(\partial M)},
$$
 there is an $\eta>0$ so that if
$$ \| f_1-R_{g_0}\|_{L^{p}(M)}+\| f_2-H_{\gamma_0}\|_{W^{1/2,p}(\partial M)}<\eta,
$$ then there exist a solution $g_1 = g_0 + \mathcal{S}^*u$ of $\Psi(g_1) =f.$ Using elliptic regularity and a bootstrap argument we have that if $f$ smooth, then $u$ is smooth.
\end{proof}

Let $M$ be a manifold with boundary and $\rho: M\to\mathbb{R}$ is a smooth function.
For $p>n,$ we set 
$$\mathcal{M}_{\rho}^{2,p}=\{g\in \mathcal{M}^{2,p};\quad R_g=\rho\quad\mbox{and}\quad H_{g_{|_{T(\partial M)}}}=0\}$$ 
and 
$$\mathcal{M}_{\tilde\rho}^{2,p}=\{g\in \mathcal{M}^{2,p};\quad R_g=0\quad\mbox{and}\quad H_{g_{|_{T(\partial M)}}}=\tilde\rho\},$$ where $\tilde\rho:\partial M\to\mathbb{R}$ is a smooth function. The sets $\mathcal{M}_{\rho}^{2,p}$ and  $\mathcal{M}_{\tilde\rho}^{2,p}$ are the set of metrics of prescribed scalar curvature  and prescribed mean curvature of the boundary, respectively. It follows from Proposition \ref{prop1}, Proposition \ref{prop2} and Theorem \ref{teols}  the following result:

\begin{corollary}
If $\rho$ and $\tilde\rho$ are not identically zero nor positive constants, then $\mathcal{M}_{\rho}^{2,p}$ and  $\mathcal{M}_{\tilde\rho}^{2,p}$ are  smooth submanifolds of $\mathcal{M}^{2,p}.$\end{corollary}

In the next proposition  we state a result about nonsurjectivity.

\begin{proposition} Let $(M,\overline g)$ be a Riemannian manifold with boundary. The following assumptions imply that $\mathcal{S}_{\overline g} $ is not surjective:
\begin{itemize}
\item[a)]     $M$ is scalar flat with totally geodesic boundary $\partial M$ or  $M$ is Ricci-flat with minimal boundary;

\item[b)]   $M=\mathbb{B}^{n+1}$ is the Euclidean ball of radius $r_0$ in $ \mathbb{R}^{n+1}$ or $M=\mathbb{S}_+^n$ is the standard hemisphere of radius $r_0$ in $ \mathbb{R}^{n+1}.$
\end{itemize}
\end{proposition}

\begin{proof}
We easily see that in the conditions of item $a),$ $Ker \mathcal{S}_{\overline g} $ is composed of constant functions and  $\mathcal{S}_{\overline g}$ is not surjective. 

If $M$ is a standard euclidean unit ball $B,$ the Steklov eigenfunctions $f$  with first nonzero eigenvalue  $n-1$ also satisfies    
\begin{equation*}
\Hess_{\overline{g}} f=0\;\;\;\mbox{in}\;\;\;M\;\;\;\;\;\mbox{and}\;\;\;\;\;\;\;
\frac{\partial f}{\partial \nu}\overline{\gamma}=\frac{f}{r_0^2}\overline{\gamma} \;\;\;\mbox{on}\;\;\;\partial M.
 \end{equation*}
Therefore,  $\mathcal{S}_{\overline g}^*$ is not surjective. Here,  functions  satisfying  $\Delta f=0$ in $M$ and $\frac{\partial f}{\partial \nu}=\frac{f}{r_0^2}$ on  $\partial M$ belongs to $ \mbox{Ker}\;\mathcal{S}_{\overline g}^*$.

In contrast, when $M=\mathbb{S}^{n}_+$ with the metric $\overline{g},$ the Robin eigenfunctions $f$ with first nonzero eigenvalue  $n$ also satisfies    
\begin{equation*}
\Hess_{\overline g} f=-\frac{f}{r_0^2}\overline{g}\;\;\;\mbox{in}\;\;\;M\;\;\;\;\;\mbox{and}\;\;\;\;\;\;\;
\frac{\partial f}{\partial \nu}\overline{g}=0 \;\;\;\mbox{on}\;\;\;\partial M.
 \end{equation*}
Analogously,  $\mathcal{S}_{\overline g}^*$ is not surjective and  all function $f\in W^{2,s}(M)$ satisfying   $\Delta f=-(n/r_0^2)f$ in $M$ and $\frac{\partial f}{\partial \nu}=0$ on $\partial M$ belongs to $ \mbox{Ker}\;\mathcal{S}_{\overline g}^*$.
\end{proof}

\section{Approximation Lemma}\label{pres}

In this section, inspired by  Theorem 2.1 in \cite{KW1},  we show how to  approximate a function arbitrarily closely in $L^p(M)$ and $W^{\frac{1}{2},p}(\partial M)$ in order to apply  Theorem \ref{teols}. We proof the following lemma.
\begin{lemma}[Approximation Lemma]\label{appr} Let $M$ be a Riemannian manifold with boundary of dimension $n\geq2.$ 
\begin{itemize}
\item[a)] Let $f, g\in C^{\infty}(\partial M).$ If the range of $g$ is in the range of $f$, that is, $\min f\leq g(x)\leq\max f $ on $\partial M$, then given any positive $\varepsilon$ there is a diffeomorphism  $\varphi$ of $M$ such that, for $p>2n,$ we have  that
$$\|f\circ \varphi-g\|_ {W^{\frac{1}{2},p}(\partial M)}<\varepsilon.$$
\item[b)] Let $f, g\in C^{\infty}( M).$ If   the range of $g$ is in the range of $f$, that is, $\min f\leq g(x)\leq\max f $ on $ M$, then given any positive $\varepsilon$ there is a diffeomorphism  $\varphi$ of $M$ such that, for $p>n,$ we have  that 
$$\|f\circ \varphi-g\|_ {L^{p}( M)}<\varepsilon.$$
\end{itemize} 
\end{lemma}

\begin{proof}
For part $a),$ let $\{\Delta_i\}$ be  a locally finite triangulation of $\partial M$ so fine that $g$ is nearly constant
in each simplex. Then we can assume that  for each $i$ we have
\begin{equation}\label{e1}
\max_{x,y\in\Delta_i}|g(x)-g(y)|<\delta
\end{equation}
where $\delta= \varepsilon/4(4 \tau vol (\partial M))^{1/p},$ where $\tau$ is a constant chosen later. Let $b_i\in int (\Delta_i)$.  By continuity there exist disjoint open sets $V_i\subset \partial M$, such that
 \begin{equation}\label{e2}
|f(x)-g(b_i)|< \delta
 \end{equation}
 for each $i$ and $x\in V_i.$ 

 Choose a neighborhood $Q$ of the $(n-1)$-skeleton $\partial M$, disjoint from  
 $b_i,$ so small that
 \begin{equation}\label{e4}
 (\max_{\partial M}|f|+\max_{\partial M}|g|)^p\cdot vol Q<\frac{\varepsilon^p}{2^{p+3}\cdot vol(\partial M)}
  \end{equation}
and 
 \begin{equation}\label{e5}
 (\max_{\partial M}|\nabla f|+\max_{\partial M}|\nabla g|)^p\cdot vol Q<\frac{\varepsilon^p}{2^{p+3}\cdot vol(\partial M)}.
 \end{equation}
 
 Consider for each $b_i$ a neighborhood $U_i$ disjoint from $Q,$  and choose open
sets  $O_1$ and $O_2$, such that
$$
\partial M-Q\subset O_1\subset \overline{O_1}\subset O_2\subset \overline{O_2}\subset \partial M-\mbox{skeleton}.
$$
We will find a diffeomorphism $\varphi$ of $M$ so that $\varphi(O_1\cap\Delta_i)\subset V_i.$ 
Firstly,  there is a diffeomorphism $\varphi_1$ of $M$ such that 
 $\varphi_1(U_i)\subset V_i,$    a diffeomorphism
 $\varphi_2$ of $M$ satisfying $\varphi_2(O_1\cap \Delta_i)\subset U_i$ and $\varphi_2|_{\partial M-O_2}= id$ for each
$i$. This allow us to define $\varphi=\varphi_1\circ\varphi_2$. Note that we are not interested in the behavior of the diffeomorphism in the interior of $M.$

Recall that 
$$\|u\|^p_{W^{\frac{1}{2},p}(\partial M)}:=\|u\|_{L^p(\partial M)}^p+[u]_{W^{\frac{1}{2},p}(\partial M)},$$  where  $[u]^p_{W^{\frac{1}{2},p}(\partial M)}:=\int_{\partial M}\int_{\partial M}\frac{|u(x)-u(y)|^p}{|x-y|^{n+\frac{p}{2}}}$ is  the Glagliardo norm of $u.$

We use (\ref{e1}), (\ref{e2}) and (\ref{e4}) to infer 

\begin{eqnarray*}
\|f\circ\varphi-g\|_{L^p}^p &=&(\int_Q+\int_{\partial M-Q})|f\circ\varphi-g|^p\\
& < &\frac{\varepsilon^p}{2^{p+3}}+\sum_i\int_{O_1\cap\Delta_i}|f\circ\varphi(y)-g(b_i)+g(b_i)-g(y)|^p \nonumber\\ 
& <& \frac{\varepsilon^p}{2^{p+3}}+\sum_i2^{p}\delta^p vol(\Delta_i)=\Big(\frac{\varepsilon^p}{2^{p+3}}+\frac{\varepsilon^p}{2^{p+2}\tau}\Big),\nonumber 
\end{eqnarray*}
where  we have omitted the volume forms. 
The next step is to study the following identity. 
\begin{equation*}
[f\circ\varphi-g]_{W^{\frac{1}{2},p}(\partial M)} =(\int_Q+\int_{\partial M-Q})\int_{\partial M}\frac{|(f\circ\varphi-g)(x)-(f\circ\varphi-g)(y)|^p}{|x-y|^{n+\frac{p}{2}}}.
\end{equation*}

For  $p>2n,$ setting $z=x-y$ we obtain  that

\begin{align*}
&\hspace{1cm}\int_Q\int_{\partial M}\frac{|(f\circ\varphi-g)(x)-(f\circ\varphi-g)(y)|^p}{|x-y|^{n+\frac{p}{2}}} \\
&\hspace{1.5cm} \leq \int_Q \int_{\partial M\cap |z|\geq 1}|2(\max_{\partial M}|f\circ \varphi|+\max_{\partial M}|g|)|^p\frac{1}{|x-y|^{n+\frac{p}{2}}}\\ 
&\hspace{1.5cm} +\int_Q\int_{\partial M\cap |z|< 1} \frac{|\max_{\partial M}|\nabla(f\circ \varphi)||x-y|+\max_{\partial M}|\nabla g||x-y||^p}{|x-y|^{n+\frac{p}{2}}}\\
&\hspace{1.5cm}\leq \int_Q 2^p(\max_{\partial M}|f\circ \varphi|+\max_{\partial M}|g|)^p\int_{\partial M\cap |z|\geq 1}\frac{1}{|z|^{n+\frac{p}{2}}}\\ 
&\hspace{1.5cm} +\int_Q (\max_{\partial M}|\nabla(f\circ \varphi)|+\max_{\partial M}|\nabla g|)^p\int_{\partial M\cap |z|< 1}|z|^{\frac{p}{2}-n}\\
&\hspace{1.5cm}< \frac{\varepsilon^p}{8} +\frac{\varepsilon^p}{2^{p+3}},
\end{align*}
where we have used   the mean value theorem in the third line, (\ref{e4}) in the fourth line and (\ref{e5}) in the fifth line.  Note that in the last line we have used the integrability of the kernel $\frac{1}{|z|^{n+\frac{p}{2}}}$  for $n+\frac{p}{2}>n.$

On the other hand, 
\begin{align*}
&\hspace{2cm}\int_{\partial M-Q}\int_{\partial M}\frac{|(f\circ\varphi-g)(x)-(f\circ\varphi-g)(y)|^p}{|x-y|^{n+\frac{p}{2}}}\\
& \hspace{0.2cm} \leq\int_{\partial M\setminus Q}\int_{|z|\geq r_0 }\frac{|f\circ\varphi(x)-g(b_i)-g(x)+g(b_i)-f\circ\varphi(y)+g(b_i)-g(y)-g(b_i)|^p}{|x-y|^{n+\frac{p}{2}}}\\ 
&\hspace{1cm} +\int_{\partial M\setminus Q} \int_{|z|< r_0}\frac{|\max_{\partial M}|\nabla(f\circ \varphi)||x-y|+\max_{\partial M}|\nabla g||x-y||^p}{|x-y|^{n+\frac{p}{2}}}\\ 
& \hspace{1cm} \leq\sum_i4^{p}\delta^p \mbox{vol}(\Delta_i)\int_{\partial M\cap |z|\geq r_0}\frac{1}{|z|^{n+\frac{p}{2}}}\\ 
&\hspace{1cm} +\int_{\partial M-Q} (\max_{\partial M}|\nabla(f\circ \varphi)|+\max_{\partial M}|\nabla g|)^p\int_{\partial M\cap |z|< r_0}|z|^{\frac{p}{2}-n}\\ 
&\hspace{1cm} < \frac{\varepsilon^p}{4}\frac{\mbox{vol}( \partial M)}{\tau r_0^{n+p/2}}+\frac{\varepsilon^p}{16},
\end{align*} 
where  we have used in the fourth line (\ref{e1}) and (\ref{e2}) and  the mean value theorem in the third line. Here, the   fifth line is less than $\frac{\varepsilon^p}{16}$ provided a constant $r_0$ can be chosen  sufficiently small such that $$r_0\cdot vol( \partial M\cap(|z|< r_0))\leq\frac{2^{-4}\varepsilon^p}{(\max_{\partial M}|\nabla(f\circ \varphi)|+\max_{\partial M}|\nabla g|)^p \mbox{vol}( \partial M)}.$$
Therefore, we see after a few calculation that if $\tau>4\frac{\mbox{vol}( \partial M)}{r_0^{n+p/2}},$ then  we have that $\|f\circ \varphi-g\|_ {W^{\frac{1}{2},p}(\partial M)}<\varepsilon.$


 The remaining item follows from an easy modification in the argument  of Theorem 2.1 of Kazdan and Warner \cite{KW1}.

\end{proof}

\begin{remark}
It would be interesting to know if Lemma \ref{appr} can be generalized to include both curvatures at the same time in order to obtain results as in the spirit of \cite{CR,LMR}, where it was considered the problem of prescribing the Gaussian and geodesic curvature of compact surfaces with boundary via  conformal change of the metric.
\end{remark}


\section{Prescribing curvature: Proof of Theorem  \ref{teoa} and \ref{teob}}\label{presc}

In this section, using approximation lemma \ref{appr} and Theorem \ref{teols},  we prescribe the scalar (resp. Gaussian, if $\mbox{dim}\;M=2)$ curvature  and  mean (resp. geodesic, if $\mbox{dim}\;M=2)$ curvature of the boundary of a certain class of  manifolds with boundary. More precisely, we show the following result. 

\begin{proposition}\label{propfund} Assume  $n\geq2$. Let $(M^n,g_0)$ be a compact Riemannian manifold with boundary.
\begin{itemize} 
\item[a)] Let $M$ be a  scalar flat manifold with mean curvature  on the boundary equal to  $H_0,$ and let $H$ be a smooth function  on $\partial M.$
If there is a constant $c>0$ satisfying 
$$ 
\min cH \leq H_0\leq \max cH.
$$Then $H$ is the mean curvature of the boundary of some scalar flat metric on $M.$
\item[b)]  Let $M$ be a  manifold with constant scalar curvature $R_0$ and minimal boundary, and let  $R$ be a smooth function on $ M.$
If there is a constant $c>0$ satisfying 
$$
\min cR \leq R_0\leq \max cR. 
$$Then $R$ is scalar curvature  of some metric with minimal boundary.
\end{itemize}
\end{proposition} 
 
 An immediate consequence is the following. 

\begin{corollary}\label{corfund} Assume  $n\geq2$. Let $(M^n,g_0)$ be a compact Riemannian manifold with boundary.
\begin{itemize} 
\item[a)] If $M$ is  scalar flat manifold and its boundary has  constant mean curvature  $H_{g_0} = H_0.$ 
Then any function $H$ on $\partial M$  having the same sign of $H_0$ somewhere is mean curvature of the boundary of some scalar flat metric on $M$, while if
$H_0\equiv 0$, then any function $H$  that changes sign is the mean curvature  of some  scalar flat metric.
\item[b)] If $M$  is a manifold with constant scalar curvature $R_{g_0}=R_0$ and minimal boundary.   Then any function $R$ on $M$ having the same sign of  $R_{0}$ somewhere is scalar curvature  of some metric with minimal boundary, while if
$R_0\equiv 0,$ then any function $R$ that changes sign is the scalar curvature of some metric with minimal boundary with respect to this metric.
\end{itemize}
\end{corollary}

\begin{proof}[Proof of Proposition \ref{propfund}]  
We will prove  item $a)$ since the other case is entirely analogous. 
If $Ker \mathcal{S}_{g_0}^*=0$ (for example, by Proposition \ref{prop1}, this may occur if $H_0$ is negative), then by Lemma \ref{appr} there  is a diffeomorphism  $\varphi$ of $M$ such that for $p>2n$ we have 
$$
\| 0-R_{g_0}\|_{L^{p}(M)}+\| c(H\circ\varphi)-H_0\|_{W^{1/2,p}(\partial M)}=\| c(H\circ\varphi)-H_0\|_{W^{1/2,p}(\partial M)}<\eta.
$$
In view of  Theorem \ref{teols} there is a metric $g_1$ satisfying $$\Psi(g_1)=(0,c(H\circ\varphi)).$$ By Lemma 52 of Cox \cite{Cox} and the  diffeomorphism invariance of scalar curvature, we have that the required metric  is  given by $g=(\varphi^{-1})^*(cg_1)$ at each point in $M$.
However, if $Ker \mathcal{S}_{g_0}^*\neq 0$ (which by   Proposition \ref{prop1}, says that $H_0$ is constant),  one may 
perturb $g_0$ slightly in order to have non constant mean curvature $H_{g_0}$  and scalar flat metric in $M$ still satisfying  
$\min cH \leq H_0\leq \max cH.$ In order to obtain such a perturbation we may consider, for example, the following change of metric $\tilde{g}= \tilde u^{\frac{4}{n-2}}g_0,$ where  $\tilde u$ is a  harmonic extension of a  function $u$ defined on the boundary that is nearly equal to $1.$\footnote{In order to obtain a metric with non constant scalar curvature $R_{g_0}$  and minimal boundary still satisfying  
$\min cR \leq R_0\leq \max cR,$ we may consider a sufficiently small perturbation of the form $\tilde{g}=g_0+h$ where  $h(X,Y)=0$ for all $X,Y\in T(\partial M).$} Hence, we can repeat the previous argument.
\end{proof}

We have now all ingredients to prove Theorem \ref{teoa} and  \ref{teob}.


\begin{proof}[Proof of Theorem \ref{teoa} and  \ref{teob} ]
 The proof goes along  the lines of Proposition \ref{propfund}. 
 Here the Osgood, Phillips and Sarnak uniformization theorem \cite{OPS}  (see also Brendle \cite{Br,Br1})  plays a fundamental role in our proof, because  there is a unique uniform flat metric  with constant  geodesic curvature boundary and  a unique uniform metric of constant curvature metric with geodesic boundary. 
\end{proof}

%
%

\section{Existence of metrics with constant curvature}\label{restr}

Let $(M^n, g)$, $n\geq3$ be a complete, n-dimensional Riemannian manifold with boundary $\partial M$. Throughout this section $R_g$ will denote the scalar curvature with respect to the Riemannian metric $g,$ while $H_g$ will be the mean curvature on the boundary.
Denote by $\tilde{g}=u^{\frac{4}{(n-2)}}g$ a metric conformally related to $g$,  where $u$ is a smooth positive function.
It is well known that the transformation law for the scalar curvature and mean curvatures are given by:
\begin{eqnarray*}
R_{\tilde{g}}&=&-\frac{4(n-1)}{(n-2)} \frac{\mathcal{L}_gu}{u^{(n+2)/(n-2)}}\quad \mbox{in}\quad M\\
H_{\tilde{g}}&=&\frac{2}{(n-2)} \frac{B_gu}{u^{n/(n-2)}} \quad\quad\quad\quad \mbox{on}\quad \partial M,
 \end{eqnarray*}
where $R_{\tilde{g}}$ and $H_{\tilde{g}}$ are the scalar curvature of and the mean curvature with respect to $\tilde{g},$ $\mathcal L=\Delta-\frac{n-2}{4(n-1)}R_g$ is the so-called conformal Laplacian and  $\mathcal B=\frac{\partial }{\partial \nu}-\frac{n-2}{2}H_{g}$   is an associated boundary operator.
 
The quadratic form associated with the operator $(\mathcal L,\mathcal  B)$ is
\begin{equation*}
E_g(u)=\int_{M}\Big(\frac{4(n-1)}{n-2}|\nabla u|_g^2+R_gu^2\Big)dv +2\int_{\partial M}H_{g}u^2d\sigma,
\end{equation*}
where $dv$ and $d\sigma$ are the Riemannian measure on $M$ and the induced Riemannian measure on $\partial M,$ respectively, with respect to the metric $g.$

 For $a,b>0$ let us define the following functional

\begin{equation}\label{yama}
Q^{a,b}_g(u)=\frac{E_g(u)}{a(\int_{M}u^{\frac{2n}{n-2}}dv)^{\frac{n-2}{n}}+b(\int_{\partial M}u^{\frac{2(n-1)}{n-2}}d\sigma)^{\frac{n-2}{n-1}}}.
\end{equation}

The Yamabe invariant is defined by
\begin{equation}\label{yamainv}
Q_g^{a,b}(M,\partial M)=\inf\{Q_g^{a,b}(u); \textrm{ $u>0$ in $C^{\infty}(M)$}\}
\end{equation}
which is invariant under conformal change of the metric $g$ for $(a,b)\in\{(0,1), (1,0)\}$ (see \cite{E1,E2}). It is not difficult to show that $Q_g^{1,0}(M,\partial M)\leq Q^{1,0}(\mathbb S^n_+,\partial \mathbb S^n_+)$ (resp. $Q_g^{0,1}(M,\partial M)\leq Q^{0,1}(\mathbb B,\partial \mathbb B)$). In \cite{E2}, Escobar  proved that if $Q_g^{1,0}(M,\partial M)\leq Q^{1,0}(\mathbb S^n_+,\partial \mathbb S^n_+)$ (resp. $Q_g^{0,1}(M,\partial M)\leq Q^{0,1}(\mathbb B,\partial \mathbb B)$), then there exists a smooth metric $u^{\frac{4}{n-2}}g$, $u>0$, of constant scalar curvature and zero mean curvature on the boundary (resp. zero scalar curvature with constant mean curvature on $\partial M$). 

Moreover we have other invariants with respect to conformal geometry  that are the eigenvalues of the boundary problem $(\mathcal L,\mathcal B)$:

\begin{equation}\label{eing1}
\left\{\begin{matrix}
\mathcal L\varphi=\lambda_1(\mathcal L)\varphi & \textrm{in $M$}\\
\mathcal B\varphi=0 & \textrm{on $\partial M$}
\end{matrix}\right.
\end{equation}
and 
\begin{equation}\label{eing2}
\left\{\begin{matrix}
\mathcal L\varphi= 0& \textrm{in $M$}\\
\mathcal B\varphi=\lambda_1(\mathcal B)\varphi & \textrm{on $\partial M$.}
\end{matrix}\right.
\end{equation}

An immediate consequence from the variation characterization of the first eigenvalue of problems (\ref{eing1}) and (\ref{eing2}) is the  following.

\begin{proposition}[Escobar \cite{E2,E3}]\label{esc}
  The first eigenfunction for problem (\ref{eing1}) or (\ref{eing2}) is strictly positive (or negative).
Moreover,  $\lambda_1(\mathcal L)$ is positive (negative, zero) if and only if $\lambda_1(\mathcal B)$ is positive (negative, zero).
\end{proposition}

We also have the following fundamental result due to  Escobar \cite{E2,E3} saying that there are three possibilities which are distinguished by the sign of the first eigenvalues  $\lambda_1(\mathcal B)$ and $\lambda_1(\mathcal L)$ (In fact, there is an analogy with (\ref{C1})). 

\begin{proposition}[Escobar \cite{E2,E3}]
Let $(M^n, g)$ be a compact Riemannian manifold with boundary $n\geq 3$. There exists a metric conformally related to g whose scalar curvature is zero and the mean curvature of the boundary does not change sign. The sign is uniquely determined by the conformal structure. Hence there are three mutually exclusive possibilities: $M$ admits a conformally related metric, which is scalar flat and of $(i)$ positive, $(ii)$ negative, or $(iii)$ identically zero mean curvature of the boundary.
\end{proposition}

It is clear that  holds an analogue result  if there exists a metric conformally related  of minimal boundary and whose  scalar  curvature does not change sign.

The Gauss-Bonnet theorem for surfaces with boundary gives obvious topological obstructions to prescribing the Gauss curvature and geodesic curvature. In fact, to best of our knowledge, there is no similar obstruction result  for $n\geq 3.$  
Taking this into account we  study  the existence of metrics 
with constant scalar curvature or constant mean curvature. 
The strategy is first  obtain metrics with $\lambda_1(\mathcal B)<0$ and $\lambda_1(\mathcal L)<0.$ The key step is to  construct a metric with negative finite total curvature 
$$
\int_MR_gdv+\int_{\partial M}H_{\gamma}d\sigma<0
$$ for certain manifolds
with boundary. Basically, we use the idea of B\'erard-Bergery \cite{BB} to  deform a metric in  a small  disk. We have  the following result.

\begin{proposition} \label{proptecn}
 Assume $n\geq 3.$  If $(M,g)$ is a manifold with smooth boundary, then either there exists a metric $g$ with $\lambda_1(\mathcal B)<0$ or  $\lambda_1(\mathcal L)<0.$
\end{proposition}

\begin{proof}
We can see from Proposition \ref{esc} that it is sufficient to prove that $\lambda_1(\mathcal L)<0.$ 
 Pick an open disk $D^n$ in $M$ and let $\overline{D}^q\times \mathbb{S}^p\subset D^n,$ where $p+q=n$ with $p\geq1$ and $q\geq 2.$  Let $f$ be a function $f$ on $\overline{D}$ depending only on the distance to origin and so that $f\equiv 1$ nearly $\partial\overline{D}.$ 

 On $\overline{D}\times \mathbb{S}^p$ we put the  warped product metric $$
g_0=f^{-\frac{2p}{n-1}}(g_d+f^2g_s),$$
where $g_s$ and $g_d$ are the standard metric on $\mathbb{S}^p$ and $\overline{D},$ respectively. 
We next consider a metric $g$ on $M$  that coincides with $g_0$ on $\overline{D}\times \mathbb{S}^p.$

The integral of scalar curvature can be written as follows: 
\begin{equation}\label{tcs}
\int_M R_gdv=\int_{M\setminus(\overline{D}\times \mathbb{S}^p)}R_gdv+\int_{(\overline{D}\times \mathbb{S}^p)}R_gdv.
\end{equation}
However, the second integral on the right hand side is equal to 
	$$
\int_{\overline{D}\times \mathbb{S}^p}R_{g_0}dv_{g_0}=\mbox{Vol}(\mathbb{S}^p,g_s)\int_{\overline{D}}	f^{\frac{p}{n-1}-2}\Big(R_{g_s}-\frac{p(n-1-p)}{n-1}|\nabla f|^2\Big)dv_{g_d}, 
	$$
since  $p\leq n-2,$ the first integral in (\ref{tcs}) and $\int_{\partial M}H_{\gamma}d\sigma$ do not depend on $f$  we can choose $f$ such that $\int_M|\nabla f|^2dv_{g_d}$ becomes sufficiently large in order to get $\int_MR_gdv$ negative as we want. Thus by variational characterization of  $\lambda_1(\mathcal L),$ we have that $\lambda_1(\mathcal L)<0$ as desired.

\end{proof}

\begin{remark}
We can still prove that if $\lambda_1(\mathcal B)<0$ or  $\lambda_1(\mathcal L)<0,$ then there is a conformal metric, $\widehat g,$ which is scalar flat  with  mean curvature $H_{\widehat g}=-1$ on the boundary  or  has
 scalar curvature $R_{\widehat g}=-1$ and minimal boundary, respectively. Indeed, assume that $\lambda_1(\mathcal B)<0,$ so we have to solve the following problem
\begin{equation}\label{prob1}
\left\{
  \begin{array}{rcl}
\Delta_gu&=&0\;\;\;\mbox{in}\;\;\;M\\
\frac{2}{n-2}\frac{\partial u}{\partial \nu}-H_{g}u&=&-u^{\beta}\;\;\;\mbox{on}\;\;\;\partial M,
\end{array}
  \right. 
\end{equation}
 where $\beta$ is the critical exponent $n/(n-2).$ Let $\varphi$ be the corresponding associated eigenfunction of (\ref{eing2}). Now, choose constants $0 < c_-< c_+$  such that $0<-\lambda_1(\mathcal B)<(c_+\varphi)^{\beta-1}$ and $-\lambda_1(\mathcal B)>(c_-\varphi)^{\beta-1}.$ Thus if 
 $u_{\pm}=c_{\pm}\varphi,$ we have $\lambda_1(\mathcal B)c_+\varphi =\mathcal B_gu_+\geq-u^{\beta}_+$ and $\lambda_1(\mathcal B)c_-\varphi =\mathcal B_gu_-\leq-u^{\beta}_-$. By the sub- and super-solutions  methods (cf. \cite{S}, Theorem 3.3) the result is clear. 
The other case are left to the reader.

\end{remark}

The following proposition  implies  the existence of metrics, depending on the case, having zero scalar curvature and minimal  boundary.

\begin{proposition}\label{propzero}
 Assume $n\geq 3.$  Let  $M^n$ be a compact manifold   with smooth boundary.   
If $M$ admits a scalar flat metric of positive mean curvature (resp. positive scalar curvature and minimal boundary), then it admits a scalar flat metric with zero mean curvature. 
\end{proposition}
\begin{proof}
 Assume that $M$ admits a scalar flat metric of positive mean curvature on the boundary. By supposition, $\lambda_1(\mathcal B,g_+)>0.$
 On the other hand,  it follows from Proposition \ref{proptecn}  that there exists $g_-$  such that is scalar flat  and has constant negative mean curvature.  So we have that  
$Q_{g_-}^{0,1}(M,\partial M)<0$. By hypothesis, there exists $g_+$ such that $Q_{g_+}^{0,1}(M,\partial M)>0.$  
Setting 
$$g_t = tg_-+ (1-t)g_+,$$  there exists $t_0\in(0,1]$  such that $0<Q_{g_{t_0}}^{0,1}(M,\partial M)<Q^{0,1}(\mathbb B,\partial \mathbb B)$  provided  $Q_{g}^{0,1}(M,\partial M)$ depends continuosly\footnote{The proof is similar to the proof that the Steklov eigenvalues $\lambda(g)$  depend continuously on $g$}	 on $g$. Therefore, the result follows from Proposition 2.1 of \cite{E2}.
\end{proof}

Combining Proposition \ref{proptecn} and \ref{propzero} we arrive at  the following proposition.

\begin{proposition}\label{tecn}
Let $M^n$, $n\geq 3,$ be a manifold with smooth boundary. 
\begin{itemize} 
\item[a)] $M$ carries a scalar flat metric with constant negative mean curvature (resp. constant negative scalar curvature with minimal boundary).
\item[b)] If $M$ carries a scalar flat metric $g$ whose boundary has mean curvature $H_{\gamma}\geq 0$ and $H_{\gamma}\neq 0$ (resp. $R_g\geq 0$ and $R_g\neq 0$ with minimal boundary), then there exists on $M$  a scalar flat metric with  mean curvature $H_g\equiv 1$ on $\partial M$ (resp. a metric that has scalar curvature$R_g\equiv 1$ and minimal boundary) and a scalar flat metric with zero mean curvature on $\partial M$. 
\end{itemize}
\end{proposition}

Finally, combining Proposition \ref{propfund},  Proposition \ref{tecn} and Proposition \ref{esc}, we can draw the following conclusion:

\begin{corollary}[Theorem \ref{corA} and Theorem \ref{corB}]
Let $M^n$, $n> 2$, be a manifold with smooth boundary.
\begin{itemize}
\item[i)]Any function that is negative somewhere on $\partial M$ (resp. on $M$) is a mean curvature  of a scalar flat metric (resp.  a scalar curvature a metric  whose boundary has mean curvature zero).
\item[ii)]Every  smooth function on $\partial M$ (resp. on $\partial M$) is a mean curvature  of a scalar flat metric (resp. a scalar curvature  of a metric, where the mean curvature of the boundary is zero) if and only if $ M$ admits a scalar flat metric
with positive constant mean curvature on the boundary (resp. positive constant scalar curvature and minimal boundary). 
\end{itemize}
\end{corollary}

\section{Proof of Theorem \ref{classes} }\label{final}

Because Theorem \ref{corA} and Theorem \ref{corB},  we can be divide the class of the compact manifolds with boundary as follows:
\begin{itemize}
\item $M$ carries a scalar flat metric $g$ whose boundary has mean curvature $H_{\gamma}\geq0$ and $H_{\gamma}\not\equiv 0$ (resp. a metric with scalar curvature $R_g\geq0$ and minimal boundary).
\item $M$ carries no scalar flat metric $g$ with positive mean curvature (no metric with positive scalar curvature and minimal boundary), but do have one with 
$H_{\gamma}\equiv 0$  and $R_g\equiv 0$.
\item $M$ carries a scalar flat metric  whose mean curvature on $\partial M$  is negative
somewhere (metric with negative scalar curvature and whose boundary is minimal).
\end{itemize}

Next we will prove Theorem \ref{classes} which follows from  the topological restrictions  discussed in Section \ref{restr}. 

%

\begin{proof}[Proof of Theorem \ref{classes}]
The result is an immediate consequence of Proposition \ref{propfund} and Proposition \ref{tecn}. 
However, it remains to show that if $M$ does not admit a scalar flat metric with positive mean curvature on the boundary, then any scalar curvature metric  with  zero mean curvature on the boundary has totally geodesic boundary (Similarly, we can prove the analogous if  $M$ does not admit a metric with  positive scalar curvature on the boundary and minimal boundary).

Assume that $\lambda_1(\mathcal B)=0.$ 
Let $g(t)$ be a smooth family of metrics with $g(0) = g$  and  infinitesimal variation $\frac{\partial}{\partial t}g(t)|_{t=0}=-\Pi_{g}.$ If  $\Pi_{g}\neq0,$ then it follows from  Proposition \ref{varieig} in Appendix A that
 \begin{equation}\label{qqcoisa}
 \frac{d}{dt}\lambda_1(\mathcal B)\Big|_{t=0}=\int_{M} |\Pi_{g}|^2d\sigma>0,
 \end{equation}
 where we have used  that $g$ is  scalar flat and has constant mean curvature on the boundary which implies that $\psi=1$ and $\Pi_{g}=\Pi_{\hat g}.$   So $\lambda_1(g(t))>0$ for all $t>0$ sufficiently small and  we conclude from Proposition \ref{esc} that there is a metric with positive mean curvature, which is a contradiction unless $\Pi_{g}\equiv 0$ holds. 
\end{proof}

\section{Appendix A}

Let $(M^n,g)$, $n\geq 3,$ be a  n-dimensional Riemannian manifold with boundary $\partial M\neq\emptyset.$ 
Recall  the following conformal boundary operator  $(\mathcal L,\mathcal B),$ where $\mathcal L=\Delta-\frac{n-2}{4(n-1)}R_g$ in $M$ and 
$\mathcal B=\frac{\partial }{\partial \nu}-\frac{n-2}{2}H_{\gamma}$ on $\partial M$.
Recall the associated  eigenvalue problems:
\begin{equation}\label{eing3}
\left\{\begin{matrix}
\mathcal L\varphi=\lambda_1(\mathcal L)\varphi & \textrm{in $M$}\\
\mathcal B\varphi=0 & \textrm{on $\partial M$}
\end{matrix}\right.
\end{equation}
and
\begin{equation}\label{eing4}
\left\{\begin{matrix}
\mathcal L\varphi= 0& \textrm{in $M$}\\
\mathcal B\varphi=\lambda_1(\mathcal B)\varphi & \textrm{on $\partial M$.}
\end{matrix}\right.
\end{equation}

Our next result concerns the variation of the first eigenvalues  of the boundary problems (\ref{eing3}) and (\ref{eing4}).

\begin{proposition}\label{varieig}
Let $\varphi$  and $\psi$ be normalized first eigenfunction of $(\mathcal L,\mathcal B)$   with respect to the interior and the boundary condition, respectively. Then 
\begin{itemize}
\item[a)] $\frac{d}{dt}\lambda_1(\mathcal L)\Big|_{t=0}=-\int_{M}\varphi^2 \langle h,Ric_{\overline{g}}\rangle dv;$ 
\item[b)] $\frac{d}{dt}\lambda_1(\mathcal B)\Big|_{t=0}=-\int_{M}\psi^2 \langle h, \Pi_{\tilde{\gamma}}\rangle d\sigma;$
\end{itemize}
 where $\overline{g}=\varphi^{4/(n-2)}g$ and $\tilde{g}=\psi^{4/(n-2)}g.$ 
\end{proposition}

\begin{proof}
Let $\varphi(t)$ and $\psi(t)$ denote the first eigenfunctions  associated to the first eigenvalues $\lambda_1(\mathcal L_t,g(t))$ and $\lambda_1(\mathcal B_t,g(t)),$ respectively. 
Taking a derivative of $\mathcal L_t\varphi(t)=\lambda_1(\mathcal L_t,g(t))\varphi(t)$ and $\mathcal B_t\psi(t)=\lambda_1(\mathcal B_t,g(t))\psi(t)$ we get that 
\begin{equation}
\mathcal L_t'\varphi(t)+\mathcal L_t\varphi'(t)=\lambda_1'(\mathcal L_t,g(t))\varphi(t)+\lambda_1(\mathcal L_t,g(t))\varphi'(t)
\end{equation}
\begin{equation}
\mathcal B_t'\psi(t)+\mathcal B_t\psi'(t)=\lambda'_1(\mathcal B_t,g(t))\psi(t)+\lambda_1(\mathcal B_t,g(t))\psi'(t)
\end{equation}
where the prime denotes derivatives with respect to $t.$

Then setting $t=0$ and using the divergence theorem we obtain that 
$$\frac{d}{dt}\lambda_1(\mathcal L)\Big|_{t=0}=\int_{M} \langle \varphi(0),\mathcal L_0'\varphi(0)\rangle dv$$
 and $$\frac{d}{dt}\lambda_1(\mathcal B)\Big|_{t=0}=\int_{\partial M} \langle \psi(0),\mathcal B_0'\psi(0)\rangle d\sigma.$$

The results now follow  by a straightforward computation using (\ref{prob}) and the following identities:
\begin{equation}
\Big(\frac{d}{dt}\Delta_t\big|_{t=0}\Big)f=-\langle h,Hess_g\;f\rangle-\frac{1}{2}\langle \nabla(\tr h),\nabla\;f\rangle+\omega(\nabla f),
\end{equation}

\begin{equation}\label{confricci}
\Ric_{\widehat{g}}=\Ric_g -(n-2)\Hess_g\;f +(n-2)|\nabla f|^2-(\Delta f+(n-2)|\nabla f|^2)g
\end{equation}
and 
\begin{equation}\label{confsecond}
\Pi_{\widehat{g}}=e^f\Pi_{g}+\frac{\partial}{\partial \nu}(e^f)g,
\end{equation}
where $\widehat{g}=e^{2f}g.$ We remark that these conformal change formulae (\ref{confricci}) and (\ref{confsecond}) can be found  at \cite{E1}.

\end{proof}



\bibliographystyle{amsplain}

\begin{thebibliography}{10}

\bibitem{A} S. Almaraz, \textit{An existence theorem of conformal scalar flat metrics on manifolds with boundary}. Pacific J. Math. 248 (2010), no. 1, 1-22.

\bibitem{HA} H. Araujo,  \textit{ Critical Points of the Total Scalar Curvature plus Total Mean Curvature Functional}. Indiana University Mathematics Journal 52, (2003), no.1, 85-107.

\bibitem{BB} L. Berard-Bergery,  \textit{ La courbure scalaire de vari\' eet\'es riemanniennes}.
S\'eminaire Bourbaki, 1980, Springer Lecture Notes, 842 225-245.


\bibitem{Br} S. Brendle,  \textit{Curvature flows on surfaces with boundary,} Math. Ann. 324 (2002), no. 3, 491-519.

\bibitem{Br1} S. Brendle,  \textit{A family of curvature flows on surfaces with boundary,} Math. Z. 241 (2002), no. 4, 829-869. 

\bibitem{B} S. Brendle, S. Chen,  \textit{An existence theorem for the Yamabe problem on manifolds with boundary}.  J. Eur. Math. Soc. 16 (2014), no. 5, 991-1016.

\bibitem{C}  O. Chodosh, \textit{Notes on linearisation stability,}  https://web.math.princeton.edu/$\sim$ ochodosh/LinStabNOTES.pdf, 2013.

\bibitem{Cox} G. Cox, \textit{Scalar curvature rigidity theorems for the upper hemisphere}. Dissertation, Duke University.(2011)

\bibitem{CR}  S. Cruz-Bl\'azquez,  D. Ruiz, \textit{Prescribing Gaussian and Geodesic Curvature on Disks}. Advanced Nonlinear Studies, 18 (2018), no. 3, 453-468.

\bibitem{ES} Y.V. Egorov,  M.A. Shubin,  \textit{Linear partial differential equations. Foundations of the classical theory.} 
Partial differential equations-1, Itogi Nauki i Tekhniki. Ser. Sovrem. Probl. Mat. Fund. Napr., 30, VINITI, Moscow, (1988), 5-255 


 \bibitem{E1} J. F. Escobar,  \textit{The Yamabe problem on manifolds with boundary}. J. Diff. Geom., 35 (1992), 21-84.

\bibitem{E2}J. F. Escobar, \textit{Conformal deformation of a Riemannian metric to a scalar flat metric with constant mean curvature on the boundary}. Annals of Mathematics, 136 (1992), 1-50.


\bibitem{E3} J. F. Escobar, \textit{Conformal deformation of a Riemannian metric to a constant scalar curvature metric with
constant mean curvature on the boundary}. Indiana Univ. Math. J. 45 (1996), 917-943.

\bibitem{E4} J. F. Escobar, \textit{Conformal metrics with prescribed mean curvature on the boundary}. Calc. Var. Partial Differential Equations, 4 (1996), 559-592.

\bibitem{EG}J. F. Escobar,  G. Garcia, \textit{ Conformal metrics on the ball with zero scalar curvature and prescribed mean curvature on the boundary}. J. Funct. Anal., 211(1) (2004), 71-152.


\bibitem{F-M} A. Fischer, J.Marsden, \textit{Deformations of the scalar curvature}.  Duke Mathematical Journal. 42 (1975), no. 3, 519 - 547. 

\bibitem{F-M2} A. Fischer, J. Mardsen, \textit{Linearization stability of nonlinear partial differential equations}.
Proc. Symp. Pure Math. 27, Part 2 (1975), 219-263.


\bibitem{H}L. H\"ormander,  \textit{The Analysis of Linear Partial Differential Operators III: Pseudo-Differential Operators.} Springer-Verlag, 2007 (1985), ISBN 978-3-540-49937-4

\bibitem{KW} J. Kazdan,  F. Warner, \textit{Curvature functions for compact 2-manifold}. Ann. of Math.,
99 (1974), 14-47.

\bibitem{KW1}  J. Kazdan,  F. Warner,\textit{ Existence and conformal deformation of metrics with prescribed Gaussian and scalar
curvature.} Ann. of Math. 101 (1975), no. 2, 317-331.



\bibitem{KW2} J. Kazdan,  F. Warner,\textit{Scalar Curvature and conformal deformation of
Riemannian structure.} J. Differ. Geom. no. 10 (1975) 113-134.

\bibitem{KW3} J. Kazdan,  F. Warner,\textit{ F.- A direct approach to the determination of Gaussian
and scalar curvature functions.} Invent. Math. no. 28 (1975) 227-230.


\bibitem{LM} J. L. Lions,  E. Magenes, \textit{Non-homogeneous boundary value problems and applications I.}
- Springer Verlag, Berlin-Heidelberg-New york., 1972.

\bibitem{LMR} R. López-Soriano, A. Malchiodi, D. Ruiz, \textit{Conformal metrics with prescribed Gaussion and geodesic curvatures.}	arXiv:1806.11533.

\bibitem{LR}  R. López-Soriano, D. Ruiz, 
 \textit{ Prescribing the Gaussian curvature in a subdomian of $\mathbb{S}^2$ with Neumann boundary condition.}
 J. Geom. Anal. 26 (2016), no. 1, 630-644.
 
 
\bibitem{M1} F. Marques, \textit{Existence results for the Yamabe problem on manifolds with boundary.} Indiana
Univ. Math. J. 54 (2005), 1599-1620.

\bibitem{M2} F. C. Marques, \textit{Conformal deformations to scalar-flat metrics with constant mean curvature
on the boundary.} Comm. Anal. Geom. 15 (2007), no. 2, 381-405.

\bibitem{MN}M. Mayer, C.B. Ndiaye, \textit{Barycenter technique and the Riemann mapping problem of Cherrier-Escobar.} J. Differential Geom. 107 (2017), no. 3, 519-560.

\bibitem{OPS}B. Osgood, R. Phillips,  P. Sarnak, \textit{Extremals of determinants of Laplacians}. J. Funct. Anal. 80 (1988), 148-211


\bibitem{S}F. Schwartz, \textit{ The zero scalar curvature Yamabe problem on noncompact manifolds with boundary}. Indiana Univ. Math. J. 55 (2006), 1449-1459.


\end{thebibliography}

\end{document}